\documentclass{article}
\usepackage{amsmath,amssymb,amsthm, algorithm, algorithmic, enumerate,listings,marvosym,hyperref,eucal,verbatim,amscd}

\title{Modular Cauchy kernel corresponding to the Hecke curve}
\author{Nina Sakharova\thanks{The article was prepared within the framework of a subsidy granted to the HSE by the Government of the Russian Federation for the implementation of the Global Competitiveness Program.  }}
\date{}

\newtheorem{theorem}{Theorem}
\newtheorem{lemma}{Lemma}

\theoremstyle{definition}

\theoremstyle{remark}
\newtheorem{remark}{Remark}

\usepackage[pdftex]{graphicx}

\usepackage{bm}
\usepackage{amsfonts}
\usepackage{amsmath}
\usepackage{amssymb}
\usepackage{amsthm}

\renewcommand{\pmod}[1]{(\textmd{mod}\hspace{.5mm}#1)}

\begin{document}

\maketitle

\begin{center}
\large{ \textit{National Research University Higher School of Economics, Russian Federation \\ Laboratory of Mirror Symmetry, 6 Usacheva str., Moscow, Russia, 119048.\\
saharnina@gmail.com}}
\end{center}
\section{Introduction}

By definition, put
$$\mu_{\gamma}(z_1, z_2)=cz_1z_2+dz_2+az_1+b,$$
where $z_1,z_2$ are arbitrary points  from the upper-half plane $\mathfrak{H} = \left\{z: \Im(z) \geq 0\right\}$ and $\gamma=\left(\begin{matrix}
	a&b\\
	c&d\\
	\end{matrix}\right)$ is an integer matrix
	with the determinant~$m$. By $\Gamma_0(N)$ we denote the Hecke congruence subgroup of $SL_2(\mathbb{Z})$, $$\Gamma_0(N)=\left\{\gamma=\left(\begin{matrix}
	a&b\\
	c&d\\
	\end{matrix}\right)| ~~c \equiv 0 \pmod{N}, ~ \mathrm{det}~\gamma =1\right\},~~~~~~~\Gamma_0(1)=\Gamma=\mathrm{SL}_2(\mathbb{Z})$$ and let $$\Gamma_{\infty} := \left\{ \left(\begin{matrix}
	1&b\\
	0&1\\
	\end{matrix}\right) \in \Gamma \right\}.$$

In this paper, we introduce the modular Cauchy kernel~ $\Xi_N(z_1, z_2)$, i.e. the modular invariant function of two variables with the asymptotics $\Xi_N(z_1, z_2) \sim \dfrac{1}{z_1-z_2}$ as $z_1 \rightarrow z_2$. Here we use this function in two cases. If the genus of the group $\Gamma_0(N)$ is greater than 0, we prove the analogue of the Zagier theorem (\cite{Za3},~ \cite{La})  for the cusp forms of weight 2. In genus zero case, we give an elementary proof of the Borcherds formula  (\cite{Bo}) for the infinite product of the difference of two normalized Hauptmoduls, ~$J_{\Gamma_0(N)}(z_1)-J_{\Gamma_0(N)}(z_2)$.\\

In (\cite{Za3},~ \cite{La}), using the Rankin-Selberg method, Don Zagier proved that the Hecke operator $T_k(m)$ on the space of cusp forms of weight $k>2$ can be defined by a kernel  $\omega_m(z_1,\bar{z_2}, k)$ as follows:
\begin{equation}
\omega_m(z_1,\bar{z_2}, k)= \sum_{ad-bc=m} \frac{1}{\mu_{\gamma}(z_1, -\bar{z_2})^k}=\sum_{ad-bc=m} \frac{1}{\left(cz_1\bar{z_2}+d\bar{z_2}-az_1-b\right)^{k}},
\end{equation}
where the sum is taken over all integer matrices	
	$\left(\begin{matrix}
	a&b\\
	c&d\\
	\end{matrix}\right)$ with the determinant~$m$.\\
\begin{theorem} [D. Zagier]
Let $\Phi_1$ be a fundamental domain for the modular group $\Gamma$ in $\mathbb{H}$ and let  \\ $C_k=\dfrac{(-1)^{k/2} \pi}{2^{k-3}(k-1)};$ then, for every holomorphic cusp form  $f$ of weight $k>2$, we have
$$\int_{\Phi_1}f(z_1)\omega_m(z_1,\bar{z_2}, k)(\Im(z_1))^{k-2} dz_1\bar{dz_1}=f\ast \overline{\omega_m( z_1, \bar{z_2})}=C_k m^{1-k}(T_k(m)f)(z_2), $$ where $f \ast g = \int_{\Phi}f(z) \overline{g(z)} (\Im(z))^{k-2} dz\bar{dz}$ is a Peterson scalar product.
\end{theorem}

In Section 2  we obtain a proof of generalization of the Zagier theorem. Namely, we prove that there is the analogous integral representation for the Hecke operators on the space of weight 2 cusp form with respect to the $\Gamma_0(N)$. Let $$\mathrm{M}(m, N)= \left\{\left(\begin{matrix}
	a&b\\
	c&d\\
	\end{matrix}\right)| ~ad-bc=m,~~(a, N)=1, ~~ c\equiv 0 \pmod{N}\right\}.$$ Consider the series \begin{equation} \omega_{m, N}(z_1, z_2)=\frac{1}{2} \sum_{\gamma \in \mathrm{M}(m, N)}\frac{1}{\mu_{\gamma}(z_1, -z_2)^2},  \label{ZagmN}
\end{equation} which is not absolutely convergent. We define the convergence as the value of this series in terms of the limit.
\begin{theorem}\label{thm1}
Let $\Phi_N$ be a fundamental domain of the modular group $\Gamma_0(N)$ in $\mathbb{H}$ and let genus of the group $\Gamma_0(N)$ be greater than zero. Then, for every holomorphic cusp form  $f$ of weight 2, we have
\begin{equation} \int_{\Phi_N}f(z_1)\omega_{m, N}(z_1,\bar{z_2})dz_2\overline{dz_2}=f\ast \overline{\omega_{m, N}(z_1, \bar{z_2})}=2 \pi i m^{-1}~(T_2(m)f)(z_1). \label{TN}\end{equation}
 \end{theorem}
From the standpoint of differential geometry, a more natural way to prove this theorem is to contract the analogue of the Cauchy kernel.
 Formally, it will be expressed by the following series
\begin{equation}
\Xi_{N, k}(z_1,z_2)=\frac{1}{2}~ \sum_{\gamma \in \Gamma_0(N)}\dfrac{1}{\mu_{\gamma}(z_1, -z_2)^k\mu_{\gamma}(z_1, -\bar{z_2})^k}. \label{Xi}
\end{equation} If $k=1$, then the series $\Xi_{N, k}(z_1,z_2)$ does not converge absolutely; but it is just at the edge of convergence, since $\sum \limits_{\gamma \in \Gamma_0(N)} \frac{1}{|\mu_{\gamma}(z_1,- z_2)|^{s}|\mu_{\gamma}(z_1,-\bar{z_2})|^{s}}$ converges for any $s$ such that $\Re(s)>1$.\\

Following the lead of E. Hecke, we investigate the series
\begin{equation}
\Xi_N(z_1,z_2, s)=~\frac{1}{2} \sum_{\gamma \in \Gamma_0(N)} \frac{\overline{\mu_{\gamma}(z_1, -z_2)}~ \overline{\mu_{\gamma}(z_1, -\bar{z_2})}}{|\mu_{\gamma}(z_1, -z_2)|^{2s}~|\mu_{\gamma}(z_1, -\bar{z_2})|^{2s}}, \label{Xi2}
\end{equation}
	where $s$ is a complex number. We will prove that there is no pole if $s=1$, and  then we will put  $\Xi_N(z_1,z_2)=\lim_{s \rightarrow 1}\Xi_N(z_1,z_2, s)$ and prove Theorem \ref{thm1}.\\

The second part of the present paper is focused on the genus zero case.

Let  $p=e^{2\pi i z_1}$, $q=e^{2\pi i z_2}$, and let $J_{\Gamma_0(N)}(p)$  be the normalized generator of the function field of $\Gamma_0(N)$. At the beginning of Section 3, we will receive the following result:
\begin{theorem}\label{thm2} Let $$E_{2, N}^{\infty}(z) = \frac{1}{2} \sum_{ \Gamma_{\infty} \backslash \Gamma_0(N)} \frac{ 1}{(cz +d)^k}$$ be a standard Eisenstein series in the cusp $i \infty$. Then
\begin{equation}\left(\Xi_N(z_1,z_2)(z_2-\bar{z_2})-2\pi i~ E_{k, N}^{\infty}(z_1)\right)~dz_1 = d_{z_1} \log~ \left|J_{\Gamma_0(N)}(p)-J_{\Gamma_0(N)}(q) \right|^2. \label{T2}
\end{equation}
 \end{theorem} 	

During the 1980s, Koike, Norton, and Zagier independently proved the following remarkable formula for an infinite product for the difference of two normalized Hauptmoduls of the group $\Gamma$:
 \begin{equation} J_{\Gamma}(p)-J_{\Gamma}(q)=p^{-1} \prod_{\substack{r>0 \\ r' \in \mathbb{Z}}} (1-p^m q^m)^{c(rr')},  \label{Bo} \end{equation} where $J_{\Gamma}(p)=j(p)-744=\frac{1}{p} + \sum_{n>0}c(n) p^n$.

 In Section 3 we will obtain a Fourier expansion for the Cauchy kernel $\Xi_N (z_1, z_2)$ and, as a corollary of the Theorem \ref{thm2}, an analogue of the formula (\ref{Bo}) for the difference of two normalized Hauptmoduls of the genus zero group $\Gamma_0 (N)$. Let $p_{r', N}(r)$ -- is the n-th Fourier coefficient of the Poincare series with the parameter $r'>0$ given by the formula (\ref{Poincare}).

 \begin{theorem}\label{thm3}  \begin{equation} J_{\Gamma_0(N)}(p)-J_{\Gamma_0(N)}(q)=\left(\frac{1}{p}-\frac{1}{q} \right) \prod_{\substack{r>0 \\ r' >0}}  \prod_{d |(r, r', N)} \left(1-p^rq^{r'} \right)^{p_{1, N}\left(rr'/d^2\right)\cdot d^2/rr'}.  \label{Bo1} \end{equation}
 \end{theorem}

We omit some technical details and calculations, since they are quite standard.  Some of the results and some calculations can be found in the preprint \cite{Sa}. \\

Despite some similarity in the formulas, it should be noted that our approach differs from that of Jan Hendrik Bruinier discussed in  \cite{Br1}. During the work on the text of this paper, we found the work \cite{BrKa} of K. Bringmann, B. Kane, S. Lobrich, K. Ono, L. Rolen  related to the very similar issues.\\

The author expresses her deep gratitude to Professor  Andrew Levin
for his inspiring guidance, help and ideas.

\section{The Cauchy kernel $\Xi_N(z_1,z_2, s)$ and the Zagier theorem}

Since the series (\ref{Xi}) does not converge absolutely, we define the sum of the series as the value of the function $\Xi_N(z_1, z_2, s) $ at the point $s = 1$. For this, we prove the following theorem.
    \begin{theorem}\label{thm3} The series $\Xi_N(z_1,z_2, s)$ can be analytically continued to the point  $s=1$. \label{T3}
\end{theorem}
\emph{Proof.}\\
\begin{enumerate}
\item Splitting the sum (\ref{Xi2}) into subsums with respect to the various values of $c$ and combining each summand with its negation, we get
$$\Xi_N(z_1,z_2, s)=\Xi_N^{0}(z_1,z_2, s)+2~\Xi_N^{>0}(z_1,z_2, s),$$
where the series $\Xi_n^{0}(z_1,z_2, s)$ corresponds  to $c=0$ and the series $\Xi_n^{>0}(z_1,z_2, s)$ corresponds  to $c>0$.

		
\item Case 1. If $c=0$, then either $a=d=1$ or $a=d=-1$, and summation over $b \in \mathbb{Z}$ is unrestricted. Therefore,
\begin{equation}
\Xi_N^{0}(z_1,z_2, s)=2\sum_{ b> 0}\frac{(\bar{z_2}-\bar{z_1}-b)~(z_2-\bar{z_1}-b)}{|z_2-z_1-b|^{2s}~|\bar{z_2}-z_1-b|^{2s}}. \label{Xi0}
\end{equation}

We have $\frac{(\bar{z_2}-\bar{z_1}-b)^n(z_2-\bar{z_1}-b)^n}{|z_1-z_2-b|^{2s}|\bar{z_2}-z_1-b|^{2s}} = \frac{1}{b^{4s-2n}}+O\left(\frac{1}{b^{4s-2n+1}} \right)$ as $b\rightarrow \infty$; hence, the sum (\ref{Xi0})	is absolutely convergent in the half-plane $\Re(s)	
>\frac{3}{4}$. \\

 \item  Case 2, $c > 0$.

Firstly, note that  $$\mu_{\gamma}(z_1,z_2)=c^{-1}\left[(cz_1+d)(cz_2-a)+(ad-bc)\right]=c^{-1}\left[(cz_1+d)(cz_2-a)\right]+1/c.$$

It is easy to check that if the series
\begin{equation}
\tilde{\Xi}_N^{>0}(z_1,z_2, s)=\sum_{\substack{c> 0 \\ c \equiv 0 \pmod{N}}}\sum_{\substack{a,d,\\ad \equiv 1 \pmod c}}\frac{(\overline{\mu_{\gamma}(z_1,-z_2)-1/c})~ \overline{\mu_{\gamma}(z_1,-\bar{z_2})-1/c})}{|\mu_{\gamma}(z_1,-z_2)-1/c|^{2s}~ |\mu_{\gamma}(z_1,-\bar{z_2})-1/c|^{2s}}
\end{equation}
can be analytically continued to some point, then the sum	$\Xi_N^{>0}(z_1,z_2, s)$ can be analytically continued to this point as well.

Therefore, we will consider the following sum
\begin{equation}
\tilde{\Xi}_N^{>0}(z_1,z_2, s)=\sum_{\substack{c> 0 \\ c \equiv 0 \pmod{N}}}\sum_{\substack{a,d,\\ad \equiv 1 \pmod c}}c^{4s-2}\frac{1}{|cz_2-a|^{4s-2}}\frac{(c\bar{z_1}+d)^{2}}{|cz_1+d|^{4s}}. \label{Xit}
\end{equation}

 Consider the classical series \cite{We1}
$$S_n(z,y,s)=\sum_{\nu}^{*}\frac{(\bar{z}+\nu)^n}{|z+\nu|^{2s}}e^{-\nu y},$$
where $n$ is an integer, $y$ is a real number, $s$ is a complex number, and  $\sum \limits^{*}$ denotes the sum taken over all integers $\nu \neq -z$.
 			
 Assume that $a=-a_0+ck$, $d=d_0+lc$, ~$k, l \in \mathbb{Z}$. If $a_0d_0 \equiv 1~ \pmod c$, then $(a_0, c)=1$. Hence, the right-hand side of (\ref{Xit}) equals
$$ \tilde{\Xi}_N^{>0}(z_1,z_2, s)=\sum_{\substack{c> 0 \\ c \equiv 0 \pmod{N}}}\sum_{\substack{1< a_0 < c\\(a_0, c)=1 \\a_0d_0 \equiv 1 \pmod c}}\frac{1}{c^{4s-2n}}S_0\left(z_2+\frac{a_0}{c},~ 0,~ 2s-n \right)S_{2}\left(z_1+\frac{d_0}{c},~ 0,~ 2s \right). \label{SS}$$

The sum ~$\tilde{\Xi}_N^{>0}(z_1,z_2, s)$~ satisfies the periodicity property $\tilde{\Xi}_N^{>0}(z_1+\nu,z_2+\nu', s)=\tilde{\Xi_N^{>0}}(z_1,z_2, s)$ for ~$\nu,\nu' \in \mathbb{Z}$, and hence $\tilde{\Xi}_N^{>0}(z_1,z_2, s)$ has a Fourier expansion of the following form
\begin{equation}
\tilde{\Xi}_N^{>0}(z_1,z_2, s)=\sum_{\substack{c> 0 \\ c \equiv 0 \pmod{N}}}\sum_{\substack{1< a_0 < c\\(a_0, c)=1 \\a_0d_0 \equiv 1 \pmod c}}\sum_{\substack{r \in \mathbb{Z}\\ r'\in \mathbb{Z}}}a(r, r', s, c)e^{2\pi i (rz_1+r'z_2)}, \label{Fourier}
\end{equation}

where the coefficients $a(r, r', s, c)$ are the products of the corresponding Fourier coefficients of $S_0\left(z_2+\frac{a_0}{c},~ 0,~ 2s-n \right)$ and $S_{2}\left(z_1+\frac{d_0}{c},~ 0,~ 2s \right)$. The calculation of the Fourier expansion of $S_n (z, y, s)$, expressed in terms of modified Bessel functions of the second kind (the MacDonald function), can be found in \cite{We2} (VII, 11). Letting $s \rightarrow$ 1, we can receive  for $ r, r '> 0 $:
\begin{equation*}
\lim_{s \rightarrow 1} \sum_{\substack{c> 0 \\ c \equiv 0 \pmod{N}}}\sum_{\substack{1< a_0 < c\\(a_0, c)=1 \\a_0d_0 \equiv 1 \pmod c}}a(r, r', s, c)= -\sum_{\substack{c> 0 \\ c \equiv 0 \pmod{N}}}4\pi^4 rr'  \frac{K(r, r'; c)}{c^2} , \label{Arr'}
 \end{equation*} where the function $K(a,b; n)=\sum\limits_{\substack{1\leq m< n, \\(m,n)=1,\\ mm^\ast\equiv 1 \pmod n}}e^{2 \pi i \left(\frac{am}{n}+\frac{bm^\ast}{n}\right)}$ is a Kloosterman sum.

There is a well-known estimate \cite{We2},~ \cite{Ku} of the Dirichlet series for the Kloosterman sums $\sum_{c> 0}\frac{K(r, r', c)}{c^2}$, which is as follows:

\begin{lemma} 	 	
If $a\geq 1$ is fixed, then for $\Re (s)>3/4$
\begin{equation}
\sum_{n\neq 0}\frac{\left|K(a,b; n)\right|}{\left|n\right|^{2s}}\leq \sqrt{a}~ \sum_{n \neq 0} \frac{d(n)}{\left|c\right|^{2s-1/2}}, \label{Weil}
\end{equation} where $d(c)=\sigma_0(n)$ is the divisor function $d(c)=\sigma_0(n)$.
\end{lemma}

Using the Weil estimate and the Dirichlet  series, involving the divisor function,
\begin{equation}
\sum_{c=1}^{\infty}\frac{d(c)}{c^s}=\zeta(s)^2, \label{d(n)}
\end{equation}
we obtain $\sum_{c > 0}\frac{\left|K(r,r'; c)\right|}{\left|n\right|^{2}}\leq \sqrt{r} ~\zeta( 3/2)$, so the coefficient $$\lim_{s\rightarrow 1}\sum_{\substack{c> 0 \\ c \equiv 0 \pmod{N}}}\sum_{\substack{1< a_0 < c\\(a_0, c)=1 \\a_0d_0 \equiv 1 \pmod c}}a(r, r', s, c)$$ is absolutely convergent.

Other nonzero terms containing $a(0, r', s, c)$, ~$a(r, 0, s, c)$ and ~$a(0, 0, s, c)$ are computed in a similar way.

 \end{enumerate}
\begin{flushright}
  $\blacksquare$
\end{flushright}

\textit{Let us define ~$\Xi_N(z_1,z_2)$ as the value of the holomorphic function $\Xi_N(z_1,z_2, s)$ at $s=1$}:
 $$\Xi_N(z_1,z_2) = \lim_{s\rightarrow 1}~ \Xi_N(z_1,z_2, s). $$
 From the behavior of ~$\Xi_N(z_1,z_2, s)$ under modular transformations, we immediately obtain the behavior of the function $\Xi_N(z_1,z_2)$:

$$\Xi_N(\gamma z_1,z_2)= (cz_1+d)^2 ~\Xi_N(z_1,z_2), ~~~\textrm{for}~~~  \gamma=\left(\begin{matrix}
	a&b\\
	c&d\\
	\end{matrix}\right) \in \Gamma.$$

\begin{remark} In order to define the sum $\omega_{m, N}(z_1,\bar{z_2}, k)= \sum\limits_{\gamma \in \mathrm{M}(m, N)}~\dfrac{1}{\mu_{\gamma}(z_1, -\bar{z_2})^k}$ for $k=2$, let us consider the series
\begin{equation}
\lim\limits_{s\rightarrow 1}\Omega_{m, N}(z_1,\bar{z_2},s)=\lim\limits_{s\rightarrow 1}\sum \limits_{\gamma \in \mathrm{M}(m, N)}\frac{\mu_{\gamma}(\bar{z_1},-z_2)^{2}}{|\mu_{\gamma}(z_1, -z_2)|^{2s-2}|\mu_{\gamma}(z_1, -\bar{z_2})|^{2s+2}}=\omega_{m, N}(z_1, z_2).
		\end{equation}
Using the same argument as in the proof of (\ref{T3}), one can prove that the series $\Omega_{m, N}(z_1,\bar{z_2},s)$ can be analytically continued to the point~$s=1$.
Therefore, we can put $\omega_N(z_1,\bar{z_2})=\lim\limits_{s\rightarrow 1}\Omega_{1, N}(z_1,\bar{z_2},s)$.
\end{remark}

The derivatives of the ``almost holomorphic'' ~ function $\Xi_N(z_1,z_2)$ have the following form:

\begin{lemma}
$$\overline{\partial } ~\Xi_N(z_1,z_2)(z_2-\bar{z_2}) = -\frac{12~\varphi (N)}{ N^2~(z_1-\bar{z_1})^2}~d\bar{z_1} -\omega_{ N}(z_1, \bar{z_2}) ~d\bar{z_2} ,$$  where
$\varphi(N)=N~ \Pi_{p|N} \left(1+\frac{1}{p}\right)$ is the Euler function.
\end{lemma}
\begin{proof}
The computation of the derivative  with respect to $\bar{z_1}$ consists of the termwise differentiation and the calculation of the asymptotics  of the two series $$-\frac{1}{2} \sum_{\gamma \in \Gamma_0(N)}\frac{1}{|\mu_{\gamma}(z_1, -z_2)|^{2s}|\mu_{\gamma}(z_1, -\bar{z_2})|^{2s-2}}~~~\mathrm{ and} ~~~-\frac{1}{2}\sum_{\gamma \in \Gamma_0(N)}\frac{1}{|\mu_{\gamma}(z_1, -z_2)|^{2s-2}|\mu_{\gamma}(z_1, -\bar{z_2})|^{2s}}.$$ The last series are holomorphic for $\Re(s)>1$ with the simple poles at $s=1$, of residues $\dfrac{3 }{2}\dfrac{\varphi(N)}{\Im(z_1)\Im(z_2)}$. Therefore, $$\lim\limits_{s \rightarrow 1}\frac{d}{d\bar{z_1}}(z_2-\bar{z_2})^{2s-1}\Xi_N(z_1,z_2,s) = -\frac{12~\varphi(N)}{(z_1-\bar{z_1})^2}.$$\\

 The termwise differentiation by  $\bar{z_2}$ gives
  \begin{multline}
\frac{d}{d\bar{z_2}}~\Xi_N(z_1,z_2,s)(z_2-\bar{z_2})^{2s-1}=(z_2-\bar{z_2})^{2s-2}\left((1-s)\sum_{\gamma \in \Gamma_0(N)}\frac{\mu_{\gamma}(\bar{z_1},-z_2)^{2}}{|\mu_{\gamma}(z_1, -z_2)|^{2s}|\mu_{\gamma}(z_1, -\bar{z_2})|^{2s}}- s~\Omega_{1, N}(z_1, \bar{z_2},s) \right).
\end{multline}
It follows that
\begin{equation}
\lim_{s \rightarrow 1} \frac{d}{d\bar{z_2}}~2~\Xi_1(z_1,z_2,s)(z_2-\bar{z_2})^{2s-1}=-\frac{1}{2}~\sum_{\gamma \in \Gamma_0(N)}\frac{1}{\mu_{\gamma}(z_1, \bar{z_2})^2}=-\omega_N(z_1, \bar{z_2}). \label{Dz_2Xi}
\end{equation}
 By the same argument, it can be shown that the function $\omega_N(z_1,\bar{z_2})$ is holomorphic with respect to $z_1$ and antiholomorphic with respect to $z_2$.
   \end{proof}

\emph{Proof of Theorem 2.} \\
 Note that it is sufficient to prove this theorem for the case of the determinant $m=1$, since $$\omega_{m, N}(z_1, z_2)=m^{-1} T_2(m)\omega_{1, N}(z_1, z_2),$$
  where the Hecke operator $T(m)$ acts on the first variable $z_1$.\\

 If $f$ is a cusp form of weight 2 with respect to the group $\Gamma_0(N)$, then  $T_2(1) f(z) = f(z)$  and we need to check that   $$\int_{\Phi_N}f(z_2)\omega_N(z_1,\bar{z_2})dz_2\overline{dz_2}=2 \pi i f(z_1).$$
Let $\mathrm{B}(a, \varepsilon)$ be a circle of radius $\varepsilon$ centered in $z=a$. We have
  \begin{multline} \int_{\Phi_N} \omega_{ N} (z_1, \bar{z_2})f(z_2) ~dz_2 d\bar{z}_2  = \lim\limits_{\varepsilon \rightarrow 0}~ \int_{\Phi_N \backslash \mathrm{B}(z_1, \varepsilon)} \frac{\partial}{\partial \bar{z_2}} \left(\Xi_N(z_1, z_2)(z_2-\bar{z_2}) f(z_2) \right) dz_2 d\bar{z}_2  = \\ = \int_{\partial \Phi_N}  \Xi_N(z_1, z_2)(z_2-\bar{z_2}) f(z_2)~dz_2 - \lim\limits_{\varepsilon \rightarrow 0}~ \int_{\partial \mathrm{B}(z_1, \varepsilon)} \Xi_N(z_1, z_2)(z_2-\bar{z_2}) f(z_2) ~ dz_2, \end{multline} where the last equality follows from the Green's formula. The first integral is equal to zero, since the integration is over the boundary of the fundamental region and the integrand is a modular-invariant function. Note that $$\Xi_N(z_1, z_2)=\frac{1}{(z_2-z_1)(\bar{z_2}-z_1)}+\frac{1}{2} \sum_{\substack{\gamma \in \Gamma_0(N),\\ \gamma \neq I}}\frac{1}{\mu_{\gamma}(z_1, -z_2)\mu_{\gamma}(z_1, -\bar{z_2})},$$ so the second integral $\int_{\partial \mathrm{B}(z_1, \varepsilon)} \Xi_N(z_1, z_2) f(z_2)~dz_2$ is equal to $- 2 \pi i f(z_1)$ by the Cauchy formula.
\begin{flushright}
$\blacksquare$
\end{flushright}

\section{The Cauchy kernel $\Xi_N(z_1,z_2, s)$ and the Borcherds Products}

In this section we calculate the Fourier coefficients of the Cauchy kernel $\Xi_{N}(z_1, z_2)$, then we prove Theorem \ref{thm2}. As a simple corollary of the Fourier expansion, we obtain Theorem \ref{thm2}.\\

 Consider the standard nonholomorphic Eisenstein series, corresponding to the cusp $i \infty$:
      \begin{equation} E_N^{\infty}(z, s)=\frac{1}{2} \sum_{\gamma \in \Gamma_{\infty}\backslash \Gamma_0(N)}\frac{(c\bar{z}+d)^2}{|cz+d|^{2s+2}}.
      \end{equation} We define the weight 2 series as a limit: $E_{2, N}^{\infty}(z_1)=\lim_{s\rightarrow 1} ~E_N^{\infty}(z_1, s)$.

The calculation of the Fourier coefficients of the non-holomorphic Eisenstein series, as well as the proof of analytic continuation to the point $s = 1$, is rather standard and we omit it. For the non-holomorphic Eisenstein series,  this expansion in powers of $p=e^{2\pi i z_1}$ is the following:
              \begin{multline}
     \lim_{s\rightarrow 1} ~E_N^{\infty}(z_1, s)=  1-\frac{3~\varphi(N)}{\pi~ N^2 ~\Im(z_1)} - 4 \pi^2\sum_{r>0} r p^r \sum_{\substack{c \equiv 0\pmod{N} \\ c>0}} \frac{1}{c^2} \sum_{l| c,~ l|r} \mu\left(\frac{c}{l}\right)l = \sum_{r=0}^{\infty}e_{r, N} p^r, \label{Eisenst}
           \end{multline} where $\mu\left(n\right)$ is the Mobius function.

The Poincare series in the cusp $i\infty$ with a complex parameter $s$ are defined by the following formula \begin{equation} P^{\infty}_{N, r'}(z, s)=\frac{1}{2}\sum_{\gamma \in \Gamma_{\infty} \backslash \Gamma_0(N)}\frac{e^{-2 \pi i r' \gamma z_1}}{(cz+d)^2|cz+d|^{2s-2}}.
      \end{equation}

    These series have the Fourier decomposition depending on the sign of the parameter $r'$:
     \begin{lemma} Let $$\lim_{s\rightarrow 1} P_{r'}(z_1, 2, s+1)= e^{2\pi i r' z_1}+\sum_{r>0}p_{ r', N}(r)e^{2 \pi i r z_1},$$ then for $r'>0$:
       \begin{equation}
           p_{r', N}(r)=-2\pi \sqrt{\frac{r}{r'}}\sum_{\substack{c \equiv 0\pmod{N} \\ c>0}}\frac{K(-r, r',c)}{c}I_1\left(\frac{4 \pi \sqrt{rr'}}{c}\right).
           \label{Poincare}
       \end{equation} and for $r'<0$:
       \begin{equation} p_{r', N}(r)=-2\pi \sqrt{\frac{r}{r'}}\sum_{\substack{c \equiv 0\pmod{N} \\ c>0}}\frac{K(-r, r',c)}{c}J_1\left(\frac{4 \pi \sqrt{rr'}}{c}\right),\end{equation} где  $$I_{1}(z)= \frac{z}{4\pi i} \int_{c-i \infty}^{c+ i \infty} t^{-2}e^{t+z^2/4t} dt, ~~~~J_{1}(z)= \frac{z}{4\pi i} \int_{c-i \infty}^{c+ i \infty} t^{-2}e^{t-z^2/4t} dt$$ are the Bessel functions of the first kind (modified and unmodified, respectively).
     \end{lemma}

Now we will evaluate the Fourier coefficients of the function $\Xi_N(z_1,z_2)$.
 \begin{theorem}\label{thm4} Let $p=e^{2\pi i z_1}$, $q=e^{2 \pi i z_2}$, $\widetilde{q}=e^{2 \pi i \bar{z_2}}$, then
 \begin{equation} \dfrac{1}{2\pi i}~\Xi_N^{\ast}(z_1, z_2)(z_2-\bar{z_2}) =  \sum_{r=0}^{\infty} e_{r, N} ~p^r  + \sum_{r>0, r'>0} p_{r', N}(r)~p^r q^{r'} - \sum_{r>0, r'<0} p_{r', N}(r)~ p^r ~\widetilde{q}^{r'}.  \label{Fur}
\end{equation}

\end{theorem}

 \emph{Proof of Theorem 5.}

 \begin{enumerate} [(i)]
   \item Note that $\mu_\gamma(z_1, \bar{z_2})-\mu_\gamma(z_1, z_2) = (cz_1+d)(\bar{z_2}-z_2)$, therefore
       \begin{multline} \Xi_1^{\ast}(z_1, z_2)(z_2-\bar{z_2}) =
          ~\frac{1}{2} \sum_{\gamma \in \Gamma_0(N)} \frac{\overline{\mu_\gamma(z_1, z_2)}\overline{\mu_\gamma(z_1,
           \bar{z_2})}(\bar{z_2}-z_2)}{|\mu(z_1, z_2)|^{2s}|\mu_\gamma(z_1, \bar{z_2})|^{2s}}=\\
           =~\frac{1}{2} \sum_{\gamma \in \Gamma_0(N)}\frac{1}{(cz_1+d)}\left(\frac{\overline{\mu_\gamma(z_1, z_2)}}{|\mu_\gamma(z_1, z_2)|^{2s}|\mu_\gamma(z_1, \bar{z_2})|^{2s-2}}-
           \frac{\overline{\mu_\gamma(z_1, \bar{z_2})}}{|\mu_\gamma(z_1, z_2)|^{2s-2}| \mu_\gamma(z_1, \bar{z_2})|^{2s}}\right).  \label{DiffXi}
       \end{multline}
 It can be easily shown that the limit of the right-hand side of (\ref{DiffXi}) as $s$ approaches 1 equals to the limit
            \begin{equation}
              \Xi_N^{\ast}(z_1, z_2)(z_2-\bar{z_2})=\lim_{s\rightarrow 1}~\frac{1}{2}\sum_{\gamma \in \Gamma_0(N)} \frac{1}{(cz_1+d)}\left( \frac{\overline{\mu_\gamma(z_1,
              \bar{z_2})}}{|\mu_\gamma(z_1, \bar{z_2})|^{2s}} - \frac{\overline{\mu_\gamma(z_1, z_2)}}{|\mu_\gamma(z_1, z_2)|^{2s}} \right).
              \label{Xiasdiff}
          \end{equation}
   \item First we find the coefficients of the Fourier expansion of the inner series in (\ref{Xiasdiff}) in powers of $q = e^{2 \pi i z_2} $ and $ \widetilde {q} = e^ {2 \pi i \bar{z_2} } $.
  For the fixed $c, d$, with $(c, d) =1$, all the pairs of integers $a, b$, such that $ad-bc=1$,  have the form $a=a_0+nc$, $b=b_0+nd$, where $a_0$, $b_0$ are some fixed solutions. Hence,
        \begin{equation}
          \sum_{\gamma \in \Gamma_0(N)}\frac{\overline{\mu_\gamma(z_1, \bar{z_2})}}{(cz_1+d)|\mu_\gamma(z_1, \bar{z_2})|^{2s}}= \sum_{\substack{c, d \in \mathbb{Z}, \\ c \equiv 0 \pmod{N}\\ (c,d)=1}}\frac{(c\bar{z_1}+d)^2}{|cz_1+d|^{2s+2}}\sum_{n =
          -\infty}^{+\infty}\frac{\left(\bar{z_2}-\frac{a_0\bar{z_1}+b_0}{c\bar{z_1}+d}+n\right)}{\left|z_2-\frac{a_0z_1+b_0}{cz_1+d}+n\right|^{2s}}.
          \label{partFourier1}
      \end{equation}
        Using the Fourier expansion of the sum $\sum_{n \in \mathbb {Z}} \frac {(\bar {z} + n)} {|z + n|^{2s}}$ (see \cite{We1}), we get:
       \begin{equation}
       \lim_{s \to 1} \sum_{n=-\infty}^{+\infty}\frac{\left(\bar{z_2}-\frac{a_0\bar{z_1}+b_0}{c\bar{z_1}+d}+n\right)}{\left|z_2-\frac{a_0z_1+b_0}{cz_1+d}+n\right|^{2s}}=-2\pi i\left(\frac{1}{2}+ \sum_{r'>0} e^{2 \pi i r' z_2}e^{- 2 \pi i r' \frac{a_0z_1+b_0}{cz_1+d}}\right). \label{Kronec1}
       \end{equation}
From (\ref{Kronec1}) and (\ref{partFourier1}) and similar computations for the first series in (\ref{Xiasdiff}), it immediately follows that
  \begin{multline}
\Xi_N(z_1,z_2)(z_2-\bar{z_2})=2 \pi i  \lim_{s\rightarrow 1} \left(E_N^{\infty}(z_1, s) + \sum_{r'> 0} P^{\infty}_{N, r'}(z_1, s)~q^{r'}- \sum_{r'<
 0}P^{\infty}_{N, r'}(z_1, s)~\widetilde{q}^{r'} \right).  \label{XiSUM}
\end{multline}
 \end{enumerate}
\begin{flushright}
  $\blacksquare$
\end{flushright}

Using the theorem \ref{thm4} and the properties of the Poincare series, it can be obtained that the function $\Xi_N(z_1,z_2)(z_2-\bar{z_2})$, viewed as
a function of the first variable $z_1$,  has a simple pole on the curve $D_N=\left\{(z_1, z_2) \in \mathbb{H} \times \mathbb{H}|~ z_2=\gamma z_1, ~\gamma \in \Gamma_0(N)  \right\}$ and for all $z_2 \in \mathbb{H}$ it is equal to zero in all cusps of $\Gamma_0(N)$. \\

\emph{Proof of Theorem 3.}\\
Now we assume that the genus of $\Gamma_0(N)$ is zero. The normalized function $J_{\Gamma_0(N)}(p)$ is a modular function of weight 0 with a simple pole with the residue 1 in $i\infty$ and holomorphic in $\mathbb{H}$. Hence the meromorphic 1-form $ d_{z_1} \log ~ |J_{\Gamma_0(N)}(p) -J_{\Gamma_0(N)}(q)|^2 $ has simple poles at the points $z_1 =\gamma z_2$ and at the $z_1=i\infty$. Thus, the differential forms in the left-hand side and in the right-hand side of (\ref{T2}) have the same poles with the same residues. As a consequence, the difference of these differential forms is a holomorphic 1-form invariant under the modular transformation. Therefore, it is equal to zero.
\begin{flushright}
$\blacksquare$
\end{flushright}

It is not difficult to show that for $r'<0 $ the Poincare series $\lim_{s \rightarrow 1} P_{r'} (z_1, 2, s + 1)$ is a cusp form of weight 2, and therefore, for the genus zero group $\Gamma_0 (N)$ it is equal to zero. In this case we can prove the formula (\ref {Bo}) of an infinite Borcherds product.\\

\emph{Proof of Theorem 4.} \\
Let $p_{r', N}(r)$ -- is the n-th Fourier coefficient of the Poincare series with the parameter $r'>0$ given by the formula (\ref{Poincare}).

From the Fourier expansion of $\Xi_N (z_1, z_2)$ it follows that
\begin{multline}
(2 \pi i)^{-1}~\Xi_N(z_1, z_2)(z_2-\bar{z_2})- E_{N}^{\infty}(z_1, s)=\\= \sum_{r>0}^{\infty} p^{-r}q^{r} - \sum_{r, r'>0}2\pi \sqrt{\frac{r}{r'}} \sum_{\substack{c \equiv 0\pmod{N} \\ c>0}}\frac{
                K(-r,r';c)}{c}I_1\left(\frac{4\pi \sqrt{rr'}}{c}\right)p^rq^{r'}=\\ \sum_{r>0}^{\infty} p^{-r}q^{r} -  \sum_{r, r'>0}2\pi \sqrt{\frac{r}{r'}}
\sum_{\substack{c \equiv 0\pmod{N} \\ c>0}} \sum_{0<m|(r,r', c)}\frac{K\left(-\frac{rr'}{m^2},1;\frac{c}{m}\right)}{c/m}I_1\left(\frac{4\pi \sqrt{rr'}}{c}\right)p^rq^{r'}  \label{docBo}
 \end{multline}
In the last equality there was used  the Selberg identity for the Kloosterman sums \cite{Ku}:
 \begin{lemma} \begin{equation} K(r, r'; c)= \sum_{0< m| (r, r',c)} mK\left(\frac{rr'}{m^2}, 1; \frac{c}{m}\right).\end{equation}
 \end{lemma}
  If $N=dd'$, $m=d m_0$, $c/m=d' c_0$ and $r/m_0=r_0$, $r'/m_0=r'_0$, then, continuing the equality (\ref{docBo}), get
                \begin{multline} \sum_{r>0}^{\infty} p^{-r}q^{r} - 2\pi \sum_{d |N} \sum_{\substack{m_0 >0 \\ c_0 >0}} \sum_{r_0,r_0'>0}  \sqrt{\frac{r_0}{r_0'}}~ \frac{K\left(-\frac{r_0r_0'}{d^2},1; \frac{c_0 N}{d}\right)}{c_0 N/d}~I_1\left(\frac{4\pi \sqrt{r_0r'_0}}{N c_0}\right)p^{r_0m_0}q^{r'_0m_0}=\\
                =- \frac{q}{q-p} -     \sum_{d |N} \sum_{r_0,r_0'>0} \sum_{m_0 >0} \frac{d^2}{r_0'}~p_{1, N}\left(r_0r_0'/d^2\right)p^{r_0m_0}q^{r'_0m_0}=\\
                =- \frac{q}{q-p} + \frac{d}{dz_1} \sum_{r, r'>0}\sum_{d|(r, r', N)}  \frac{d^2}{rr'}~p_{1, N}\left(rr'/d^2\right) \log \left(1-p^rq^{r'} \right) =\\
\frac{d}{dz_1}\log  \left(\frac{1}{p} - \frac{1}{q} \right) \prod_{r,r'>0}\prod_{d|(r, r',N)}\left(1-p^rq^{r'} \right)^{p_{1, N}\left(rr'/d^2\right)\cdot d^2/rr'}\label{JJ}\
        \end{multline}
 This completes the proof.
 \begin{flushright}
   $\blacksquare$
 \end{flushright}


\bibliographystyle{amsplain}

\end{document}